\documentclass[11pt]{article}

\usepackage{amsmath}
\usepackage{amsthm}
\usepackage{amssymb}
\usepackage[numbers,sort&compress]{natbib}
\usepackage{mathtools}
\usepackage{hyperref}
\usepackage{xcolor}

\theoremstyle{plain}
\newtheorem{theorem}{Theorem}[section]
\newtheorem{lemma}[theorem]{Lemma}
\newtheorem{corollary}[theorem]{Corollary}

\theoremstyle{definition}

\newtheorem{example}[theorem]{Example}

\theoremstyle{remark}

\newcommand{\z}{\zeta}

\newcommand{\R}{\mathbb{R}}

\DeclareMathOperator{\Ls}{Ls}
\DeclareMathOperator{\Li}{Li}
\DeclareMathOperator{\ls}{ls}
\DeclareMathOperator{\re}{Re}
\DeclareMathOperator{\im}{Im}
\newcommand{\abs}[1]{\left|#1\right|}                
\newcommand{\bra}[1]{\left(#1\right)}                
\newcommand{\sqbra}[1]{\left[#1\right]}              
\newcommand{\set}[1]{\left\{#1\right\}}              
\newcommand{\lgs}[2]{\log\bra{\frac{\sin #1}{\sin #2}}}

\DeclareMathOperator{\lsh}{lsh}
\newcommand{\lgsh}[2]{\log\bra{\frac{\sinh #1}{\sinh #2}}}
\newcommand{\tlgsh}[2]{\log\bra{\tfrac{\sinh #1}{\sinh #2}}}

\title{\bf Closed-form evaluations of log-sine integrals and Apéry-like sums in terms of polylogarithms}
\author{Noam Shalev}

\date{}

\newcommand{\Address}{{
		\bigskip
		\footnotesize
		N.~Shalev, \textsc{Einstein Institute of Mathematics,
			The Hebrew University of Jerusalem, Israel}\par\nopagebreak
		\textit{E-mail address}: \texttt{noam.shalev2@mail.huji.ac.il}
		
}}

\begin{document}
\maketitle

\begin{abstract}
We present a new systematic method for evaluating generalized log-sine integrals in terms of polylogarithms. Our approach is based on an identity connecting ordinary generating functions of polylogarithms to integrals involving the sine function. This method provides closed-form expressions for log-sine integrals of weight up to 4 using only classical polylogarithms, while higher weights require Nielsen polylogarithms.
Later we generalize this identity and show how it gives rise to numerous Apéry-like formulae extending results of Koecher, Leshchiner and others. We also derive hyperbolic analogues and recover several functional equations between Nielsen polylogarithms. In the process, we derive new parametric identities similar to those given by Saha and Sinha.
\end{abstract}

\section{Introduction and Preliminaries}
Classically, the \emph{log-sine integrals} are defined as
\[
\Ls_m(\theta)\coloneqq -\int_0^{\theta} \log\bra{2\sin \frac{t}{2}}^{m-1} dt,
\]
where $m\in \mathbb{N}$ is the \emph{weight} of the integral and $0\leq \theta \leq \pi$. These integrals, together with their moments (\emph{generalized log-sine integrals})
\[
\Ls_m^{(k)}(\theta)\coloneqq -\int_0^{\theta} t^k\,\log\bra{2\sin \frac{t}{2}}^{m-1-k} dt,\quad \qquad(0\leq k< m),
\]
 were widely studied, tracing back to investigations by Nielsen (1906) \cite{nielsen1906} and Bowman (1947) \cite{bowman47}, and the foundational work of
 Lewin (1958) \cite{lewin58,lewin81}. Recent decades have seen renewed interest in such integrals \cite{nyw95,ccs09,bs11,choi13,orr19, cglcz22}, motivated in part by their appearance in physics applications, particularly in  higher-order terms of the $\epsilon$-expansion of Feynman diagrams \cite{fk99,kv00, dk01}.
  
  A main goal in the study of these integrals
   is to find closed-form expressions for them in terms of other well-studied special functions. These often include integer values of the \emph{Riemann zeta function} 
 $\z(s)\coloneqq \sum_{n\geq 1}n^{-s}$ ($\re s >1$),
  real/complex parts of \emph{polylogarithms} $\Li_m(x)\coloneqq \sum_{n\geq 1} x^n n^{-m}$ ($|x|< 1$), and values of generalizations of the classical polylogarithm, which we will shortly discuss. 
  Previous work on log-sine integrals has largely focused on  $\theta\in \set{\pi/3,\pi/2,2\pi/3,\pi}$ \cite{lewin81,zucker85, dk00,dk01}. In \cite{dk00} several formulas for general values of $\theta$ were given, 
  involving \emph{Nielsen polylogarithms} (see \S~\ref{sec:higher-weight} below), and in \cite{bs11, bs15} it was demonstrated that all generalized log-sine integrals admit closed-forms in terms of Nielsen polylogarithms.

The goal of this paper is to present a new method for finding closed-form expressions of log-sine integrals in terms of classical polylogarithms for low weights ($m\leq 4$), and in terms of Nielsen polylogarithms for higher weights.  We prove a simple yet powerful identity (Theorem~\ref{thm:sum-no-prod} below) connecting the ordinary generating function $f(a)$ of $\Li_m(1-e^{-2 iz})$, $m\in \mathbb{N}$, to integrals involving $(\sin t)^a$, and compare coefficients of $a^m$ to obtain identities involving log-sine integrals. We then show how to generalize this method to obtain closed-forms for all generalized log-sine integrals, recovering in the process some relations between Nielsen polylogarithms (cf. \cite{bs15, cgr21}). Our methods ultimately rely on a remarkable infinite series identity given in Theorem~\ref{thm:sum-t-1-t} (discussed also in \S~\ref{sec:further}). 

 The relative elegance of the identities we find makes a strong case for considering \emph{normalized} log-sine integrals,
\begin{equation*}
\ls_m^{(k)}(z)\coloneqq -\int_0^z t^k \log\bra{\frac{\sin t}{\sin z}}^{m-1-k}dt, \qquad \bra{0< z \leq \frac{\pi}{2}},
\end{equation*}
 which relate to the classical log-sine integrals via
 \begin{align*}
\Ls_m^{(k)}(2z) = 2^{k+1}\sum_{j=0}^{m-k-1}\binom{m-k-1}{j} \log(2\sin z)^j 
\ls_{m-j}^{(k)}(z)
 \end{align*}
(in particular, $\Ls_m^{(k)}\bra{\pi/3} =2^{k+1} \ls_m^{(k)}\bra{\pi/6}$). 
The integral 
$
-\int_0^z (z-t)^k \lgs{t}{z}^{m-1-k}dt
$
is another variant which arises naturally in this context; we refer to such integrals as \emph{shifted} normalized log-sine integrals.

Another motivation for the evaluation of log-sine integrals is their connection to Apéry-like sums. In his celebrated 1979 paper \cite{apery79}, Apéry published the first proof of the irrationality of $\z(3)$ (and a new proof of the irrationality of $\z(2)$). The starting point for his approach was the two identities
\begin{equation}\label{eq:Apery-sums}
\z(2)=3\sum_{n=1}^\infty \frac{1}{n^2 \binom{2n}{n}},\qquad
\z(3)=\frac{5}{2}\sum_{n=1}^\infty \frac{(-1)^{n-1}}{n^3 \binom{2n}{n}}.
\end{equation}

Owing to the success of Apéry's proof, much work has since focused on the evaluation of sums of the form $\sum_{n\geq 1} a_n n^{-m}\binom{2n}{n}^{-1}$, where the sequence $a_n$ often involves a combination of polynomials, exponentials, and (generalized) harmonic numbers (recall that the classical harmonic numbers are given by $H_N=\sum_{n=1}^N \frac{1}{n}$, and the generalized harmonic numbers of order $m\in \mathbb{N}$ are given by 
$H_N^{{(m)}}=\sum_{n=1}^N n^{-m}$, with the convention that $H_0^{(m)}=0$). Such sums are frequently referred to as \emph{Apéry-like} sums \cite{rivoal04,bbb06,sun15,chu20,au20,jlr23} or \emph{central binomial} sums \cite{bbk01} (due to the involvement of the central binomial coefficients). The central binomial sums $S(m)\coloneqq \sum_{n=1}^\infty n^{-m}\binom{2n}{n}^{-1}$ (which were considered in e.g. \cite{zucker85,plouffe98,dk01, bbk01}) are intimately linked with log-sine integrals via the more general identities \cite{kv00,bbk01}
\begin{align}
\label{eq:ls-m-1-series}
 \ls^{(1)}_m(z)
&=
 -\int_0^z t \lgs{t}{z}^{m-2}dt
= \frac{(-1)^{m-1}(m-2)!}{2^m}
\sum_{n=1}^\infty \frac{(2 \sin z)^{2n}}{n^m \binom{2n}{n}}, 
\\\notag
\ls^{(0)}_m(z)
&=
-\int_0^z  \lgs{t}{z}^{m-1}dt
= (-1)^{m-1}(m-1)!
\sum_{n=1}^\infty \binom{2n}{n}\frac{(\sin z)^{2n+1}}{2^{2n}(2n+1)^m},
\end{align}
with $m\geq 2$, which can be obtained, among similar expressions for other normalized log-sine integrals, using the well-known Taylor expansions of $\arcsin$-powers, 
\begin{equation}\label{eq:arcsin-even}
z^{2m}=\frac{(2m)!}{2^{2m}}\sum_{n=1}^\infty \frac{(2 \sin z)^{2n} }{n^2 \binom{2n}{n}}\sum_{0<n_1<\ldots< n_{m-1}< n-1} \frac{1}{(n_1n_2\cdots n_{m-1})^2},
\end{equation}
\begin{equation}\label{eq:arcsin-odd}
z^{2m-1}=(2m-1)!\sum_{n=0}^\infty \binom{2n}{n}\frac{( \sin z)^{2n+1} }{2^{2n} (2n+1)}\sum_{0\leq n_1<\ldots< n_{m-1}< n-1} \frac{1}{((2n_1+1)\cdots (2n_{m-1}+1))^2},
\end{equation}
valid for $0\leq z \leq \frac{\pi}{2}$
(and where the inner sums on the right-hand side can be expressed in terms of generalized harmonic numbers via the Newton-Girard formulae).

\section{Evaluation of low-weight log-sine integrals}
In this section we show how to easily obtain closed-form expressions for $\ls_m^{(k)}$ with $m\leq 4$ (and also $\ls_5^{(1)}$). We first note the trivial identity 
$\ls_m^{(m-1)}(z)= -z^m/m$ for $m\in \mathbb{N}$. We can also easily express $\ls_{m}^{(m-2)}(z)$ in terms of polylogarithms using the well-known Fourier expansion
\begin{equation}\label{eq:fourier}
\log(2 \sin t) = -\sum_{n=1}^\infty \frac{\cos(2 tn)}{n}= -\re \sum_{n=1}^\infty \frac{e^{-2 i t n}}{n},
\end{equation}
which immediately implies the following formula.
\begin{lemma}\label{lem:ls-m-m-2}
	For $0< z \leq \frac{\pi}{2}$ and $m\geq 2$,
\begin{equation*}
\ls_{m}^{(m-2)}(z)=\frac{z^{m-1}}{m-1}\log(2\sin z)+(m-2)! \re \sqbra{
\frac{ \zeta(m)}{(2i)^{m-1}} - \sum_{k=0}^{m-2} \frac{ z^{m-2-k} \Li_{k+2}(e^{-2 i z})}{(2i)^{k+1}(m-2-k)!} 
}.
\end{equation*}
\end{lemma}
Our main method is based on the following seemingly simple identity (whose generalizations we discuss in \S~\ref{sec:apery}).
\begin{theorem}\label{thm:sum-no-prod} For $0< z <\frac{\pi}{6}$ and 
 $-1<a<1$, we have
\begin{equation}\label{eq:int-exp}
\sum_{n=1}^\infty \frac{\bra{1-e^{-2 i z}}^n}{n+a} = 2 i \int_0^z \exp\bra{a\sqbra{\log\bra{\frac{\sin t}{\sin z}}+i(z-t)}}dt.
\end{equation}
\end{theorem}
\begin{proof}
Let $0< z < \frac{\pi}{6}$. For $0\leq t \leq z$ we have 
$0<|1-e^{-2 i t}|=|2 \sin t| <1 $, so we can integrate term-by-term in the following calculation:
 \begin{align*}
 2i &\int_0^z \bra{\frac{\sin t}{\sin z}}^a e^{ia(z-t)}dt
 =2i (1-e^{-2iz})^{-a}\int_0^z (1-e^{-2it})^a dt
 \\&= (1-e^{-2iz})^{-a}\int_0^z2ie^{-2i t} \sum_{n=1}^\infty (1-e^{-2it})^{n+a-1} dt
\\&=(1-e^{-2iz})^{-a} \sum_{n=1}^\infty \sqbra{\frac{(1-e^{-2 it})^{n+a}}{n+a}}_{t=0}^{t=z}=\sum_{n=1}^\infty \frac{\bra{1-e^{-2 i z}}^n}{n+a}.
 \end{align*}
 We note that in \eqref{eq:int-exp} we assume the principal branch of logarithm. 
\end{proof}

Expanding both sides of \eqref{eq:int-exp} in powers of $a$ and comparing coefficients, we obtain the following simple formula for $\Li_m(1-e^{-2 i z})$ in terms of shifted normalized log-sine integrals, which extends to $0< z < \frac{\pi}{2}$ by analytic continuation.
\begin{corollary}\label{cor:Li_n}
Let $0< z <\frac{\pi}{2}$. For $m\geq 0$, we have
\begin{equation}
(-1)^m \Li_{m+1}(1-e^{-2i z}) = \frac{2i}{m!}\int_0^z\bra{\log\bra{\frac{\sin t}{\sin z}}+i(z-t)}^m dt.
\end{equation}
\end{corollary}

\begin{example}\label{ex:Li-n} For $0< z < \frac{\pi}{2}$,
\begin{align}
\re \Li_3(1-e^{-2 iz})&= -2\int_0^z (z-t)\lgs{t}{z}dt \label{eq:re-li3}
\\
\re \Li_4(1-e^{-2 iz})&= \frac13\int_0^z \sqbra{3(z-t)\lgs{t}{z}^2- (z-t)^3}dt
\label{eq:re-li4}
\\
\re \Li_5(1-e^{-2 iz})&=- \frac13\int_0^z \sqbra{(z-t)\lgs{t}{z}^3-(z-t)^3\lgs{t}{z} }dt
\end{align}
and likewise,
\begin{align}
\im \Li_2(1-e^{-2 iz})&= -2\int_0^z  \lgs{t}{z} dt \label{eq:im-li2}\displaybreak[1]
\\
\im \Li_3(1-e^{-2 iz})&= \int_0^z \sqbra{ \lgs{t}{z}^2-(z-t)^2 }dt\label{eq:im-li3}\displaybreak[1]
\\
\im \Li_4(1-e^{-2 iz})&= -\frac13\int_0^z \sqbra{ \lgs{t}{z}^3-3(z-t)^2\lgs{t}{z} }dt\label{eq:im-li4}\displaybreak[1]
\\
\im \Li_5(1-e^{-2 iz})&= \frac{1}{12}\int_0^z \sqbra{ \lgs{t}{z}^4-6(z-t)^2\lgs{t}{z}^2+(z-t)^4 }dt\label{eq:im-li5}
\end{align}
\end{example}
Note that equations \eqref{eq:im-li2} and \eqref{eq:re-li3} respectively read $\im \Li_2(1-e^{-2 iz})=2\ls_2^{(0)}(z)$ and 
$\re \Li_3(1-e^{-2 iz}) = 2z\ls_2^{(0)}(z)-2 \ls_3^{(1)}(z)$. Comparing with the expressions for $\ls_2^{(0)}$ and $\ls_3^{(1)}$ indicated by Lemma~\ref{lem:ls-m-m-2}, we recover the following two polylogarithms identities\footnote{Equation \eqref{eq:identity-imli2} also follows from the dilogarithm reflection formula, and \eqref{eq:identity-reli3} appears in \cite[Eq. (6.54)]{lewin81}.}:
\begin{align}
\im \Li_2(e^{-2 i z})+\im\Li_2(1-e^{-2 iz}) &= 2z \log(2 \sin z),
\label{eq:identity-imli2}
\\
\re \Li_3(e^{-2 i z}) +2 \re\Li_3(1-e^{-2 i z}) &= \z(3) + 2 z^2 \log(2 \sin z).
\label{eq:identity-reli3}
\end{align}
From equation \eqref{eq:im-li3} we immediately reproduce the identity
\begin{equation}\label{eq:ls-3-0}
\ls_3^{(0)}(z)=-\int_0^z \lgs{t}{z}^2 dt = -\frac{z^3}{3}-\im\Li_3(1-e^{-2iz}),
\end{equation}
which was given (in a non-normalized log-sine form) in \cite[Eq. (6.56)]{lewin81}. The normalized form \eqref{eq:ls-3-0} was  also given in \cite{rs15} and extended by Reshetnikov \cite{rs15, cd19}
to a closed-form for the hypergeometric series 
$_4F_3(\frac12,\frac12,\frac12,\frac12; \frac32,\frac32,\frac32;z)$ for all $z\in \mathbb{C}$.

Theorem~\ref{thm:sum-no-prod} allows us to give
closed-forms in terms of classical polylogarithms for $\ls_m^{(1)}(z)$ with $m\in\{3,4,5\}$ (recall the series expansion \eqref{eq:ls-m-1-series} of $\ls_m^{(1)}(z)$):
\begin{corollary}\label{cor:ls-n-1}
We have, for $0< z < \frac{\pi}{2}$,
\begin{align}
\ls_3^{(1)}(z)&= \frac{1}{8} \sum_{n=1}^\infty \frac{(2\sin z)^{2n}}{n^3 \binom{2n}{n}} 
= \frac{z}{2} \im \Li_2(1-e^{-2 iz})-\frac12 \re \Li_3(1-e^{-2i z})
\label{eq:ls-3-1}\displaybreak[1]
\\
\ls_4^{(1)}(z)&= -\frac{1}{8} \sum_{n=1}^\infty \frac{(2\sin z)^{2n}}{n^4 \binom{2n}{n}} 
= -\frac{z^4}{4} + \re \Li_4(1-e^{-2 iz}) - z \im \Li_3(1-e^{-2 iz}) \label{eq:ls-4-1}\displaybreak[1]
\\
\ls_5^{(1)}(z)&= \frac{3}{16} \sum_{n=1}^\infty \frac{(2\sin z)^{2n}}{n^5 \binom{2n}{n}} 
=\frac{3z^4}{4}\log(2 \sin z)+\frac{3}{8}\z(5)+\frac{3 z^2}{4}\z(3) 
\label{eq:ls-5-1}
\\&\quad+\frac{3z}{4}\bra{\im\Li_4(e^{-2iz})+4\im\Li_4(1-e^{-2iz})} -\frac{3}{8}\bra{\re\Li_5(e^{-2iz})+8\re\Li_5(1-e^{-2iz})}\notag
\end{align}
\end{corollary}
\begin{proof}
Equation \eqref{eq:ls-3-1} follows directly as a combination of equations \eqref{eq:re-li3} and \eqref{eq:im-li2}, while equation
\eqref{eq:ls-4-1} follows as a combination of \eqref{eq:re-li4} and \eqref{eq:ls-3-0}. Let us prove \eqref{eq:ls-5-1}.
	Multiplying both sides of equation \eqref{eq:int-exp} by $e^{-i z a}$ and comparing the real parts of the coefficients of $a^4$, we get
\[
\int_0^z \sqbra{t\lgs{t}{z}^3 - t^3 \lgs{t}{z}  }dt =3 \re\sqbra{\sum_{i=1}^\infty (1-e^{-2i z})^n \bra{
\frac{1}{n^5}+i\frac{z}{n^4}-\frac{z^2}{2n^3}-i\frac{z^3}{6n^2}+\frac{z^4}{24n}
}}
\]
The sum on the right is clearly a combination of polylogarithms with argument $1-e^{-2 iz}$. Meanwhile, by Lemma~\ref{lem:ls-m-m-2},
the integral 
$\ls_5^{(3)}(z)=-\int_0^z t^3 \lgs{t}{z}dt$ is given as a combination of polylogarithms with argument $e^{-2i z}$. Solving for $\ls_5^{(1)}(z)$ and simplifying the expression using identities 
\eqref{eq:identity-imli2}--\eqref{eq:identity-reli3} gives \eqref{eq:ls-5-1}.
\end{proof}
 
A similar calculation, based on equation \eqref{eq:im-li4},
shows that we also have the identity
\begin{equation}\label{eq:ls-4-0}
\ls_4^{(0)}(z)=-\int_0^z \lgs{t}{z}^3 dt = \frac{3z}{2}\z(3)+z^3 \log(2 \sin z) +\frac{3}{4}\im\Li_4(e^{-2iz})+3\im\Li_4(1-e^{-2 iz}).
\end{equation}
An equivalent evaluation of $\ls^{(0)}_4(z)$ was also given in \cite{s15}, and the particular case $z=\frac{\pi}{4}$ was used by Cantarini \& D'Aurizio in \cite{cd19} to give closed forms for certain hypergeometric series and Euler sums via Fourier-Legendre expansions. Another equivalent evaluation (of the unnormalized log-sine integral $\Ls^{(0)}_4(z)$), may be found in \cite{cglcz22}.

\section{Remarks on higher-weight log-sine integrals}\label{sec:higher-weight}
The method which we have described in the previous section is successful for evaluating low-weight log-sine integrals, which turn out to be expressible using only classical polylogarithms. Closed-forms for some higher-weight integrals frequently involve \emph{multiple polylogarithms}. 
The multiple polylogarithm of one variable is defined by \cite{bbbl99, waldschmidt02}
\[
\Li_{s_1,\ldots, s_k}(x) = \sum_{ n_1 >\ldots > n_k>0} \frac{x^{n_1}}{n_1^{s_1}\cdots n_k^{s_k}}
\]
for $(s_1,\ldots,s_k)\in \mathbb{N}^k$ and $|x|<1$. By induction on the differential identity 
$d/dx \Li_{1,s_2,\ldots,s_k}(x)= \Li_{s_2,\ldots,s_k}(x)/(1-x)$ we deduce the Taylor expansion
\begin{equation}\label{eq:Li_1..1-log}
\Li_{\{1\}^k}(x)=\frac{(-1)^k}{k!}\log(1-x)^k,
\end{equation}
where $\{1\}^k$ denotes the index 1 repeated $k$ times.
Another family of functions generalizing the classical polylogarithm, introduced by Nielsen \cite{nielsen09, k86}, is given by
\begin{equation}\label{eq:nielsen-def}
S_{m,k}(x)\coloneqq \frac{(-1)^{m+k-1}}{(m-1)!k!}\int_0^1 \log(t)^{m-1}\log(1-x t)^{k} \frac{dt}{t},
\end{equation}
(the \emph{Nielsen polylogarithms}). Here we assume the principal branch of the logarithm. The Nielsen polylogarithms are in fact multiple polylogarithms:
\[S_{m,k}(x)= \Li_{m+1,\{1\}^{k-1}}(x)\]
(with the classical polylogarithms being recovered as $S_{m,1}=\Li_{m+1}$). 
For example,
$S_{m-1,2}(x)=\sum_{n\geq 1} H_{n-1}n^{-m}x^n$.
The real$\!\!$~/~$\!\!$complex parts of $S_{m-1,2}(e^{i \pi/3})$, also referred to as \emph{multiple Glaisher$\!\!$~/~$\!\!$Clausen values}, have been shown \cite{bbk01, dk01} to be connected with Apéry-like sums, log-sine integrals, and multiple zeta values. In \cite{bs11, bs15}, Borwein \& Straub showed that $\ls_m^{(k)}$ is always expressible in terms of Nielsen polylogarithms.  Here we arrive at the same conclusion by considering a generalization of Theorem~\ref{thm:sum-no-prod}, which we give in the following theorem.
 We adopt from \cite{bs11} the notation 
\[H_{n-1}^{[k]}\coloneqq \sum_{n>n_1>\ldots >n_k} \frac{1}{n_1\cdots n_k}\]
	for the \emph{multiple harmonic numbers} (with the convention that $H^{[0]}_{n-1}\coloneqq 1$).
\begin{theorem}\label{thm:sum-gen-H} Let $0< z < \frac{\pi}{6}$, 
	$-1<a<1$, and $k\geq 0$. Then
	\begin{equation}\label{eq:int-tk-exp}
	\sum_{n=1}^\infty \frac{H_{n-1}^{[k]}}{n+a}\bra{1-e^{-2 i z}}^n = \frac{(2i)^{k+1}}{k!}\int_0^z t^k\cdot\exp\bra{a\sqbra{\log\bra{\frac{\sin t}{\sin z}}+i(z-t)}}dt.
	\end{equation}
\end{theorem}
\begin{proof} 
	 By definition, $\Li_{\{1\}^{k+1}}(x)=\sum_{n\geq 1} H_{n-1}^{[k]} n^{-1}x^n$. In view of \eqref{eq:Li_1..1-log}, we have
 \begin{align*}
\frac{(2i)^{k+1}}{k!} &\int_0^z t^k \bra{\frac{\sin t}{\sin z}}^a e^{ia(z-t)}dt
=\frac{(2i)^{k+1}}{k!} (1-e^{-2iz})^{-a}\int_0^z t^k (1-e^{-2it})^a dt
\\&= (1-e^{-2iz})^{-a}\int_0^z 2i  e^{-2i t} \sum_{n=1}^\infty H_{n-1}^{[k]} (1-e^{-2it})^{n+a-1} dt
\\&=(1-e^{-2iz})^{-a} \sum_{n=1}^\infty H_{n-1}^{[k]} \sqbra{\frac{(1-e^{-2 it})^{n+a}}{n+a}}_{t=0}^{t=z}=\sum_{n=1}^\infty \frac{H_{n-1}^{[k]}}{n+a}\bra{1-e^{-2 i z}}^n,
\end{align*}
where the above calculation is valid because $|1-e^{-2iz}|<1$  for $0<z<\frac{\pi}{6}$.
\end{proof}

Comparing coefficients of $a^m$ in \eqref{eq:int-tk-exp}, we obtain a generalization of Corollary~\ref{cor:Li_n}, which again extends to $0< z < \frac{\pi}{2}$ by continuation (cf. e.g. \cite{k86} for the analytic properties of the Nielsen polylogarithms):
\begin{equation}\label{eq:Li_m+1,k}
(-1)^m S_{m,k+1 }(1-e^{-2i z}) = \frac{(2i)^{k+1}}{m!k!}\int_0^z t^k\,\bra{\log\bra{\frac{\sin t}{\sin z}}+i(z-t)}^m dt.
\end{equation}
\begin{corollary}
$\ls_m^{(k)}(z)$ is always expressible in terms of Nielsen polylogarithms. Explicitly, for $m>k\geq 0$ and $0< z< \frac{\pi}{2}$ we have
\[
\frac{\ls_m^{(k)}(z)}{(m-k-1)!}=\sum_{j=0}^{m-k-1}\sum_{r=0}^{j} \frac{(k+j-r)!}{r!\,(j - r)!} \frac{
 (-1)^{m+j + r + 1}
	(i)^{r + k + 1}}{	2^{k+j+1-r}} z^r S_{m-k-1-j,\, k+j+1-r}(1-e^{-2 iz}).
\]
\end{corollary}
\begin{proof} The claim follows immediately from \eqref{eq:Li_m+1,k} by writing
 \begin{align*}
 \ls_m^{(k)}(z)&=-\int_0^z t^k \bra{\lgs{t}{z}+i(z-t) -i(z-t)  }^{m-k-1}dt
 \\&=-\sum_{j=0}^{m-k-1}\binom{m-k-1}{j}\int_0^z t^k (-i (z-t))^{j}\bra{\log\bra{\frac{\sin t}{\sin z}}+i(z-t)}^{m-k-1-j}dt
 \end{align*}
 and expanding $(z-t)^{j}$ in the last integral by a second application of the binomial theorem.
\end{proof}
Equation \eqref{eq:Li_m+1,k} allows us to express Nielsen polylogarithms in terms of log-sine integrals of the same weight (or lower). Because we've already derived closed-forms in terms of classical polylogarithms for all log-sine integrals of weight $\leq 4$, we thus recover the fact that Nielsen polylogarithms of weight $\leq 4$ are reducible (cf. \cite[\S~3.2]{cgr21}). In particular, we deduce the following identities.
\begin{example}  For $0< z < \frac{\pi}{2}$,
\begin{align}
S_{0,2}(1-e^{-2i z}) &= -2 z^2 \label{eq:li11}
\displaybreak[1]
\\
S_{1,2}(1-e^{-2iz}) &=  -4 \ls_3^{(1)}(z)+i  \frac{2 z^3}{3}
\displaybreak[1]
\\
S_{2,2}(1-e^{-2iz}) &=  \frac{ z^4}{6}+2 \ls_4^{(1)}(z) + i \bra{ 4\ls_3^{(1)}(z)- 4z \ls_2^{(0)}(z)}\label{eq:li31}
\displaybreak[1]
\\
\re S_{3,2}(1-e^{-2i z})&=-\frac{2}{3}\ls_5^{(1)}(z)+2 \ls_5^{(3)}(z)-4 z\ls_4^{(2)}(z)+2 z^2 \ls_3^{(1)}(z)\label{eq:reli41}
\displaybreak[1]
\\
\im S_{3,2}(1-e^{-2i z})&=-\frac{z^5}{30}+2 \ls_5^{(2)}(z)-2z \ls_4^{(1)}(z)
\label{eq:imli41}
\displaybreak[1]
\\
S_{2,3}(1-e^{-2 iz}) &= 4\ls_5^{(3)}(z)-4z \ls_4^{(2)}(z) + i \bra{2\ls_5^{(2)}(z)+\frac{z^5}{15}} 
\displaybreak[1]
\\
\re S_{3,3}(1-e^{-2 iz}) &= \frac{z^6}{90}-2 \ls_6^{(3)}(z)+2 z \ls_5^{(2)}(z)
\end{align}
\end{example}
Note that all of the log-sine integrals appearing in \eqref{eq:reli41}
have been evaluated in terms of classical polylogarithms in the previous section, showing that  $\re S_{3,2}(1-e^{-2 i z})$ is reducible in terms of classical polylogarithms. In fact, as we will see in section \S~\ref{sec:nielsen-functional-equations}, this follows from a more general two-term functional equation for $S_{3,2}$ (Lemma~\ref{lem:s32}), while another three-term equation (Lemma~\ref{lem:s32-three-term}) shows that $\im S_{3,2}(1-e^{-2 i z})$ can be written as a combination of $\im S_{3,2}(e^{-2i z})$ and classical polylogarithms.
 We also note that taking $z=\pi/6$ in some of the equations above recovers several identities displayed in \cite{bbk01}.

Using Theorem~\ref{thm:sum-gen-H} and its corollaries, we may derive closed-expressions for log-sine integrals such as the following.
\begin{corollary}\label{cor:ls-5-6}
	For $0< z< \frac{\pi}{2}$, we have the identities
\begin{align}
\ls_5^{(0)}(z) &= -\frac{z^5}{5} + 3 \im S_{3,2}(1-e^{-2 iz}) -12\im \Li_5(1-e^{-2 iz}) - 6z \re \Li_4(1-e^{-2 iz})
\label{eq:ls-5-0}\displaybreak[1]
\\
\ls_5^{(2)}(z) &= -\frac{7z^5}{30} + \frac12 \im S_{3,2}(1-e^{-2 iz}) + z \re \Li_4(1-e^{-2 iz}) - z^2 \im \Li_3(1-e^{-2 iz})
\label{eq:ls-5-2}\displaybreak[1]
\\
\ls_6^{(0)}(z) &= z^5 \log(2 \sin z) +5 z^3 \z(3)-\frac{15 z}{2}\z(5) -15 \im S_{4,2}(1-e^{-2 i z})\label{eq:ls-6-0}
\\&\quad-\frac{15}{4}\bra{ \im\Li_6(e^{-2 iz})-16 \im\Li_6(1-e^{-2 iz})} +30 z \re \Li_5(1-e^{-2 i z})\notag
\end{align}
\end{corollary}
\begin{proof}
Equation \eqref{eq:ls-5-2} follows immediately from \eqref{eq:imli41} by substituting the expression \eqref{eq:ls-4-1} which we have derived for $\ls_4^{(1)}(z)$. Now, notice that combining equations \eqref{eq:im-li5}, \eqref{eq:ls-4-1} and \eqref{eq:ls-3-0} (or by the same method as in the proof of Corollary~\ref{cor:ls-n-1}), we get
\begin{align}
-\ls_5^{(0)}(z)+6 \ls_5^{(2)}(z)&=
\int_0^z \sqbra{\lgs{t}{z}^4-6 t^2 \lgs{t}{z}^2}dt \label{eq:-ls50+6ls52}
\\&\hspace{-2em}=
-\frac{6z^5}{5} + 12\im \Li_5(1-e^{-2iz}) + 12 z \re \Li_4(1-e^{-2iz})  - 
6 z^2\im \Li_3(1-e^{-2iz}),\notag
\end{align}
from which we deduce \eqref{eq:ls-5-0}. Equation \eqref{eq:ls-6-0} is similarly obtained by setting $m=4,k=1$ in \eqref{eq:Li_m+1,k}, comparing imaginary parts, and simplifying the resulting expression via identities
\eqref{eq:identity-imli2}-\eqref{eq:identity-reli3} and known evaluations of low-weight log-sine integrals.
\end{proof}
 We can thus obtain closed-forms for normalized log-sine integrals at $z=\frac{\pi}{4}$ (equivalently, for classical log-sine integrals at $\theta=\frac{\pi}{2}$) such as
 \begin{align}
\int_0^{\frac{\pi}{4}} \log\bra{\sqrt{2}\sin t}^4 dt&=\frac{\pi ^5}{5120}
 +12 \im\Li_5(1+i)-3 \im S_{3,2}(1+i)+\frac{3\pi}{2}  \re\Li_4(1+i)
  \label{eq:at-pi4-ls50}
 \\
 \int_0^{\frac{\pi}{4}} \log\bra{\sqrt{2}\sin t}^5 dt&=
 \frac{\pi^5 }{2048}\log(2) +\frac{5 \pi^3 }{64}\z(3)
-\frac{15 \pi }{8}\z(5) -15 \im S_{4,2}(1+i)
 \\&\quad+\frac{15}{4}\beta(6)+60 \im\Li_6(1+i) +\frac{15\pi}{2} \re \Li_5(1+i)\notag
 \end{align}
 which follow from equations \eqref{eq:ls-5-0}, \eqref{eq:ls-6-0} respectively, and where $\beta(m)\coloneqq \sum_{n\geq 0} (-1)^{n}(2n+1)^{-m}$ is the Dirichlet Beta function. As remarked above, in \S~\ref{sec:nielsen-functional-equations} we prove a three-term functional equation for $S_{3,2}$ which implies that
we can further simplify \eqref{eq:at-pi4-ls50} via
 \begin{align*}
 \im S_{3,2}(1+i) =\im S_{3,2}(i)+\frac{\pi}{2}\re\Li_4(1+i)
 -\frac{151 \pi^5}{46080}+\frac{\log(2)}{2}\beta(4)-\frac{\pi^3 \log(2)^2}{384}-\frac{\pi }{4}\log(2)\z(3).
 \end{align*}

\section{Generation of Apéry-like identities}\label{sec:apery}
Driven by the success of Apéry's proof, many authors have looked for identities similar to the rapidly-converging series given in equation \eqref{eq:Apery-sums}. Koecher \cite{koecher1980} (and independently Leshchiner \cite{leshchiner81}) proved that
\begin{equation}\label{eq:koecher}
\sum_{n=1}^\infty \frac{1}{n(n^2-a^2)}=\frac{1}{2}\sum_{n=1}^\infty \frac{(-1)^{n-1}}{n^3\binom{2n}{n}} \frac{5n^2-a^2}{n^2-a^2}\prod_{k=1}^{n-1}\bra{1-\frac{a^2}{n^2}},
\end{equation}
which produces many Apéry-like sums for odd integer values of the zeta function.  Almkvist \& Granville \cite{ag99} proved a similar identity for $\zeta(4m+3)$ (first conjectured by  Borwein \& Bradley \cite{bb97}), and later Bradley \cite{b02} (independently, Rivoal \cite{rivoal04}) proved a conjectural bivariate identity due to Cohen. Many more similar identities were derived by Pilehrood \& Pilehrood \cite{pp08,pp11} using Wilf–Zeilberger theory. In this section we study an identity (given in Theorem~\ref{thm:sum-t-1-t} below) generalizing the one given in Theorem~\ref{thm:sum-no-prod}, and show that it leads to a generalization of the Koecher-Leshchiner identity \eqref{eq:koecher}. In the next sections we investigate further implications of Theorem~\ref{thm:sum-t-1-t}.

To begin, we note that Theorem~\ref{thm:sum-no-prod} can be also be seen as the particular case $b=-\frac{a}{2}$ of a more general claim:
\begin{theorem}\label{thm:sum-prod}
Let 
 $-1<a<1$ and $b\in \mathbb{R}$. For $0< z < \frac{\pi}{6}$, we have	
	\begin{equation}\label{eq:sum-prod}
	\sum_{n=1}^\infty \frac{(1-e^{-2i z})^n}{n+a}\prod_{k=1}^{n-1} \frac{k-b +\frac{a}{2}}{k+a} = 2i\int_0^z \exp\bra{a \lgs{t}{z}}\exp\bra{2ib(t-z)}dt,
	\end{equation}
\end{theorem}  
We will prove Theorem~\ref{thm:sum-prod} shortly. In the meantime, we prove first the following (even more general) identity.
\begin{theorem}\label{thm:sum-t-1-t}
	Let $-1<c<0$, $-1<a<1$, and $b\in \mathbb{R}$. Furthermore, let $t\in \mathbb{C}$ be a number satisfying both $|t|< 1$ and $|t|(1-|t|)^c< |c|^c/(c+1)^{c+1}$.
	 Then
\begin{equation}\label{eq:sum-t-1-t}
\sum_{n=1}^\infty \frac{t^n}{n+a}\prod_{k=1}^{n-1}\frac{k-b+a(c+1)}{k+a}=\sum_{n=1}^\infty t^n(1-t)^{b+n c}\,\frac{
\bra{1+b+nc}_{n-1}
}{(n-1)!(n+a)} ,
\end{equation}
where $(x)_n\coloneqq \frac{\Gamma(x+n)}{\Gamma(x)}=x(x-1)\cdots(x-n+1)$ is the Pochhammer symbol (rising factorial).
\end{theorem}
\begin{proof}
Consider the formal double series
\[
\sum_{m=1}^\infty \sum_{k=0}^\infty  t^{m+k}(-1)^k \binom{b+ m c}{k} \; \frac{(1+b+mc)_{m-1}}{(m-1)!(m+a)}
\eqqcolon\sum_{m=1}^\infty \sum_{k=0}^\infty t^{m+k}C_{m,k},
\]
where $\binom{b+mc}{k}=\frac{1}{k!}\frac{\Gamma(1+b+mc)}{ \Gamma(1+b+mc-k)}$ is the generalized binomial coefficient.  By the binomial theorem, for $|t|<1$ the series
\[(1-t)^{b+mc}=\sum_{k=0}^\infty t^k(-1)^k \binom{b+mc}{k}\]
converges absolutely. Hence, by Fubini's theorem for double infinite series, for values of $t$ for which the double series absolutely converges, one may write
\[
\sum_{n=1}^\infty t^m(1-t)^{b+m c}\,\frac{
	\bra{1+b+mc}_{m-1}
}{(m-1)!(m+a)}=\sum_{m=1}^\infty \sum_{k=0}^\infty t^{m+k}C_{m,k} = \sum_{n=1}^\infty t^n \sum_{\substack{m+k=n\\m\geq1, k\geq 0}} C_{m,k},
\]
thus reducing the problem to proving the finite sum identity
\[
\sum_{m=1}^n \frac{(-1)^{n-m}}{(n-m)! (m-1)!} \, \frac{\Gamma(b+( c+1) m)}{(m+a)\Gamma(1+b+(c+1)m-n)} = \frac{1}{n+a}\prod_{k=1}^{n-1}\frac{k-b+a(c+1)}{k+a}.
\]
Indeed, the left-hand side in the above equation is just the partial fraction decomposition of the right-hand side, which can be verified by considering 
$\lim_{a\to -m} \frac{m+a}{n+a}\prod_{k=1}^{n-1}\frac{k-b+a(c+1)}{k+a}$ for each $1\leq m\leq n$. We only need to prove that the double series  is absolutely convergent when $t$ satisfies the conditions in the theorem. One may check (e.g. via Stirling's approximation)
that the power series $\sum_{m\geq 1} y^m\frac{(1+b+mc)_{m-1}}{(m-1)!(m+a)}$ has radius of convergence $|c|^c/(c+1)^{c+1}$. Since $c$ is negative, $b+mc<0$ when $m$ is sufficiently large.  By the definition of generalized binomial coefficients, for $\alpha< 0$ we have $(-1)^k\binom{\alpha}{k}>0$ for all $k\geq 0$, hence
\[\sum_{k=0}^\infty |t|^k \abs{ \binom{\alpha}{k}}  = \sum_{k=0}^\infty |t|^k (-1)^k \binom{\alpha}{k}=(1-|t|)^\alpha.\]
Therefore if $|t|<1$ and $|t|(1-|t|)^c< |c|^c/(c+1)^{c+1}$, then
$\sum_{m\geq1}\sum_{k\geq0}|t^{m+k}C_{m,k}|< \infty$ as desired.  We remark  that the hypotheses of Theorem~\ref{thm:sum-t-1-t} are sufficient for the stated conclusion, though based on numerical evidence we expect the result to hold more generally (see \S~\ref{sec:further}).
\end{proof}

\begin{proof}[Proof of Theorem~\ref{thm:sum-prod}] First we prove the theorem for the case $0< z \leq \frac{\pi}{8}$.
Setting $t\coloneqq 1-e^{-2 iz}$ with $0< z \leq \frac{\pi}{8}$, we have 
$|t|=2\sin z<1$, $t(1-t)^{-1/2}=2i\sin z$ and $|t|(1-|t|)^{-1/2}<2=(\frac12)^{-\frac12}/(\frac12)^{\frac12}$. Hence, by setting $c=-\frac12$ in Theorem~\ref{thm:sum-t-1-t}, it follows that the left-hand side of \eqref{eq:sum-prod} equals
 \begin{align*}
 e^{-2 i z b}&\sum_{n=1}^\infty \frac{(2 i \sin z)^n}{n+a} \frac{\bra{1+b-\frac{n}{2}}_{n-1}}{(n-1)!}
 =\frac{e^{-2izb}}{(2i \sin z)^{a}}\sum_{n=1}^\infty
  \frac{\bra{1+b-\frac{n}{2}}_{n-1}}{(n-1)!}\int_0^z (2i\sin t)^{n+a-1}(2i\cos t)dt
 \\&=2i\sum_{n=1}^\infty \int_0^z \bra{\frac{\sin t}{\sin z}}^a e^{-2i z b}\cos t\frac{\bra{1+b-\frac{n}{2}}_{n-1}}{(n-1)!}(2 i \sin t)^{n-1}dt.
 \end{align*}
The above calculation shows that it suffices to prove that
\[
e^{2i t b}=\cos t\sum_{n=1}^\infty \frac{\bra{1+b-\frac{n}{2}}_{n-1}}{(n-1)!}(2 i \sin t)^{n-1}
\]
for $0\leq t \leq \frac{\pi}{8}$. Indeed, the real and complex parts of the above identity are precisely the well-known hypergeometric series (cf. \cite[\S~2.8,(11)-(12)]{bateman53})
\[
\frac{\cos(2tb)}{\cos t}=\sum_{n=0}^\infty \frac{(\sin t)^{2n}}{(2n)!}\prod_{k=0}^{n-1}(2k+1)^2-(2b)^2,
\quad
-\frac{2b\sin(2tb)}{\cos t}=\sum_{n=0}^\infty \frac{(\sin t)^{2n+1}}{(2n+1)!}\prod_{k=0}^{n}(2k)^2-(2b)^2.
\] 
The extension of the claim to $0< z < \frac{\pi}{6}$ follows by analytic continuation.
\end{proof}
Theorem~\ref{thm:sum-prod} indicates that we can study log-sine integrals via infinite series. Setting $c=-\frac12$ and $b=-\frac{a}{2}$ in Theorem~\ref{thm:sum-t-1-t}, we get the following identity.
\begin{lemma} Let $-1<a<1$. For $|t|< 2(\sqrt{2}-1)$, we have
	\begin{equation}\label{eq:sum-t}
	\sum_{n=1}^\infty \frac{t^n}{n+a}=\sum_{n=1}^\infty t^n(1-t)^{-\frac{n+a}{2}} \frac{\bra{1-\frac{n+a}{2}}_{n-1}}{(n-1)! (n+a)}.
	\end{equation}
	\end{lemma}
Specializing to $t=1-e^{-2iz}$ for $0< z \leq \frac{\pi}{8}$ gives an equivalent formulation of Theorem~\ref{thm:sum-no-prod} in terms of infinite series (where the equivalence is indicated by Theorem~\ref{thm:sum-prod}), namely,
\begin{equation}\label{eq:sum-A_n}
e^{-iza}\sum_{n=1}^\infty \frac{(1-e^{-2 iz})^n}{n+a} = \sum_{n=1}^\infty \frac{(2 i \sin z)^n}{(n-1)!} A_n(a)
\end{equation}
(also valid for $0<z <\frac{\pi}{6}$ by continuation),
where we denote 
$A_n(a)\coloneqq  \frac{1}{(n+a)}\cdot\bra{1-\frac{n+a}{2}}_{n-1}$.
This formulation is reminiscent of the Koecher-Leshchiner-type identities, highlighting a deep connection between polylogarithms and Apéry-like sums. Indeed, it implies the following identities.
\begin{lemma}\label{lem:brad-like}
\begin{equation}\label{eq:brad-sin-2n-ver1}
\begin{split}
a^2\sum_{n=1}^\infty \frac{(2\sin z)^{2n}}{n\binom{2n}{n}((2n)^2-a^2)} &\prod_{k=1}^{n-1}\bra{1-\frac{a^2}{(2k)^2}} \\&=  \cos(a z)\re \sum_{n=1}^{\infty} \frac{n(1-e^{-2iz})^n}{n^2-a^2} +a\sin(az)\im  \sum_{n=1}^{\infty} \frac{(1-e^{-2iz})^n}{n^2-a^2}.
\end{split}
\end{equation}
\begin{equation}\label{eq:brad-sin-2n}
\begin{split}
2a\sum_{n=1}^\infty \frac{(2\sin z)^{2n}}{\binom{2n}{n}((2n)^2-a^2)} &\prod_{k=1}^{n-1}\bra{1-\frac{a^2}{(2k)^2}} \\&= -a \cos( a z)\re \sum_{n=1}^{\infty} \frac{(1-e^{-2iz})^n}{n^2-a^2} +\sin(a z)\im  \sum_{n=1}^{\infty} \frac{n(1-e^{-2iz})^n}{n^2-a^2}.
\end{split}
\end{equation}
\begin{equation}\label{eq:brad-sin-2n+1-ver1}
\begin{split}
2\sum_{n=0}^\infty \binom{2n}{n}\frac{(\sin z)^{2n+1}(2n+1)}{2^{2n}((2n+1)^2-a^2)} &\prod_{k=0}^{n-1}\bra{1-\frac{a^2}{(2k+1)^2}}
\\&\hspace{-4em}=\cos(a z)\im \sum_{n=1}^{\infty} \frac{n(1-e^{-2iz})^n}{n^2-a^2} -a\sin(az)\re  \sum_{n=1}^{\infty} \frac{(1-e^{-2iz})^n}{n^2-a^2}.
\end{split}
\end{equation}
\begin{equation}\label{eq:brad-sin-2n+1}
\begin{split}
2a\sum_{n=0}^\infty \binom{2n}{n}\frac{(\sin z)^{2n+1}}{2^{2n}((2n+1)^2-a^2)} &\prod_{k=0}^{n-1}\bra{1-\frac{a^2}{(2k+1)^2}}
\\&\hspace{-4em}=a\cos(a z)\im \sum_{n=1}^{\infty} \frac{(1-e^{-2iz})^n}{n^2-a^2} +\sin(az)\re  \sum_{n=1}^{\infty} \frac{n(1-e^{-2iz})^n}{n^2-a^2}.
\end{split}
\end{equation}
\end{lemma}
\begin{proof}
By writing the rising factorials in their product definition, it is straightforward to check that 
\[
A_{2n}(a)+A_{2n}(-a)= 
-\frac{\prod_{k=0}^{n-1}(a^2-(2k)^2)}{ 2^{2n-2}(a^2-(2n)^2)}
,\quad A_{2n+1}(a)+A_{2n+1}(a)=
-\frac{(2n+1) \prod_{k=0}^{n-1}(a^2-(2k+1)^2)}{2^{2n-1}(a^2-(2n+1)^2)}.
\]
By \eqref{eq:sum-A_n}, we have
 \begin{align*}
\re\sum_{n=1}^\infty (1-e^{-2 i z})^n \bra{ \frac{e^{-iza}}{n+a} + \frac{e^{iza}}{n-a}}&=\re \sum_{n=1}^\infty \frac{(2 i \sin z)^n}{(n-1)!} (A_n(a)+A_n(-a))
\\&=\sum_{n=1}^\infty \frac{(-1)^{n-1}(2\sin z)^{2n}}{(2n-1)!} (A_{2n}(a)+A_{2n}(-a)),
 \end{align*}
 which simplifies to \eqref{eq:brad-sin-2n-ver1}. Taking imaginary parts instead of real parts results in \eqref{eq:brad-sin-2n+1-ver1}. Equations \eqref{eq:brad-sin-2n}, \eqref{eq:brad-sin-2n+1} are obtained similarly by considering the identities
 \[
 A_{2n}(a)-A_{2n}(-a)= 
 \frac{n\prod_{k=0}^{n-1}(a^2-(2k)^2)}{ a 2^{2n-3}(a^2-(2n)^2)}
 ,\quad A_{2n+1}(a)-A_{2n+1}(a)=
\frac{a \prod_{k=0}^{n-1}(a^2-(2k+1)^2)}{2^{2n-1}(a^2-(2n+1)^2)}.
 \]
\end{proof}
As mentioned, identities \eqref{eq:brad-sin-2n-ver1}--\eqref{eq:brad-sin-2n+1} may be seen as generalizations the Koecher-type identities discussed above.  For example,
letting $z\to\frac{\pi}{6}$ in \eqref{eq:brad-sin-2n},
 we obtain the identity
\begin{equation*}
\sum_{n=1}^\infty \frac{1}{\binom{2n}{n}((2n)^2-a^2)} \prod_{k=1}^{n-1}\bra{1-\frac{a^2}{(2k)^2}} =
\frac{\pi a \csc\bra{\frac{\pi a}{2}} -2 \cos \bra{\frac{\pi a}{6}}}{8a ^2},
\end{equation*} 
which was given in an equivalent form by Leshchiner in \cite{leshchiner81} (also see \cite{bbb06}). 

We also note that \eqref{eq:sum-t} implies, by the same method as in the proof of Lemma~\ref{lem:brad-like}, that we have e.g.
\begin{align}\label{eq:t-apery}
\begin{split}
\sum_{n=1}^\infty t^n \bra{ \frac{(1-t)^{a/2}}{n+a}+\frac{(1-t)^{-a/2}}{n-a} }
&=2 a^2 \sum_{n=1}^\infty \frac{ (-1)^{n-1} \bra{\frac{t}{\sqrt{1-t}}}^{2n} }{n \binom{2n}{n} ((2n)^2-a^2)} \prod_{k=1}^{n-1}\bra{1-\frac{a^2}{(2k)^2}}
\\&\hspace{-3em}+2 \sum_{n=0}^\infty \binom{2n}{n} \frac{ (-1)^n \bra{\frac{t}{\sqrt{1-t}}}^{2n+1} (2n+1) }{2^{4n} ((2n+1)^2-a^2)}\prod_{0=1}^{n-1}\bra{1-\frac{a^2}{(2k+1)^2}}
\end{split}
\end{align}
We can directly compare coefficients of $a^m$ in \eqref{eq:t-apery} to obtain Apéry-like identities. For instance, comparing coefficients of $a^2$ and recalling the Taylor expansion \eqref{eq:arcsin-odd} of $\operatorname{arcsinh}^3$, we obtain
\begin{align}
 \sum_{n=1}^\infty \frac{(-1)^{n-1}\bra{\frac{t}{\sqrt{1-t}}}^{2n}}{n^3 \binom{2n}{n}} - 2 \sum_{n=0}^\infty \binom{2n}{n} \frac{(-1)^n \bra{\frac{t}{\sqrt{1-t}}}^{2n+1}}{2^{4n}(2n+1)^3}
 =4\Li_3(t)-2 \log(1-t)\Li_2(t)-\frac{\log(1-t)^3}{3}.
\end{align}
It is interesting to note that for $t=\frac{\sqrt{5}-1}{2}$ we have $\frac{t}{\sqrt{1-t}}=1$, so formula \eqref{eq:t-apery} produces many identities relating alternating Apéry-like sums and polylogarithms  involving the golden ratio (cf. \cite{lewin91}).

\section{Hyperbolic analogues}\label{sec:hyperbolic}

Identity \eqref{eq:t-apery} makes it clear that the advantage of setting $t=1-e^{-2 i z}$ is the ability to isolate any of the two series on the right-hand side by considering real/complex parts, due to the fact that $\frac{t}{\sqrt{1-t}}=2 i \sin z$. 
 Alternatively, we can isolate these series by considering the even/odd parts of the real-valued identity which is the hyperbolic analogue of \eqref{eq:Li_m+1,k}. More precisely, we have:
 \begin{theorem}\label{thm:Li_m+1,k-hyperbolic}
 	For all $z\in \R$,
 \begin{align}
 \label{eq:Li-hyper-plus}
 (-1)^m  S_{m,k+1 }(1-e^{2 z}) &= \frac{(-2)^{k+1}}{m!k!}\int_0^z t^k\,\bra{\log\bra{\frac{\sinh t}{\sinh z}}-(z-t)}^m dt,
 \\
  \label{eq:Li-hyper-minus}
  (-1)^m  S_{m,k+1 }(1-e^{-2 z}) &= \frac{2^{k+1}}{m!k!}\int_0^z t^k\,\bra{\log\bra{\frac{\sinh t}{\sinh z}}+(z-t)}^m dt,
 \end{align}
 (where for $z=0$ we understand the integrals as $0$).	
 \end{theorem} 
\begin{proof}For $z<\frac12\log(2)$ we have $|1-e^{2z}|<1$, in which case \eqref{eq:Li-hyper-plus} follows exactly as in the proof of Theorem~\ref{thm:sum-gen-H}. Since both sides of \eqref{eq:Li-hyper-plus} are real-analytic for $z\in \R$ (cf. \cite{k86}), we conclude that \eqref{eq:Li-hyper-plus} holds for all $z\in \R$ (with the convention that both sides are 0 for $z=0$). By setting $z\mapsto -z$ we obtain \eqref{eq:Li-hyper-minus}.
\end{proof}
As a corollary of Theorem~\ref{thm:Li_m+1,k-hyperbolic} we obtain a family of identities, analogous to those derived in the previous sections, involving the two-term expressions ${S_{m,k}(1-e^{2 z})\pm S_{m,k}(1-e^{-2 z})}$. Here, we denote the normalized \emph{log-sinh integrals} by
\begin{equation*}
\lsh_m^{(k)}(z)\coloneqq -\int_0^z t^k \lgsh{t}{z}^{m-1-k}dt.
\end{equation*}
\begin{corollary} Analogous to equations \eqref{eq:ls-3-1}--\eqref{eq:ls-5-1}, we have
\begin{align}
\label{eq:lsh-3-1}
\lsh_3^{(1)}(z) &= -\frac{z}{4}\bra{\Li_2(1-e^{2z})-\Li_2(1-e^{-2z}) }+\frac14 \bra{ \Li_3(1-e^{2 z})+\Li_3(1-e^{-2z})}
\displaybreak[1]
\\
\label{eq:lsh-4-1}
\lsh_4^{(1)}(z)&=\frac{z^4}{4}+\frac{z}{2}\bra{ \Li_3(1-e^{2 z})-\Li_3(1-e^{-2z})}-\frac12 \bra{ \Li_4(1-e^{2 z})+\Li_4(1-e^{-2z})}
\displaybreak[1]
\\
\lsh_5^{(1)}(z)&=\frac{z^5}{5}- \frac{3z^4}{4}\log(2 \sinh z)-\frac{3}{8}\z(5)+\frac{3 z^2}{4}\z(3) -z^3\z(2)
\label{eq:lsh-5-1}
\\&\quad+\frac{3z}{4}\bra{
\Li_4(e^{-2z})-2 \Li_4(1-e^{2 z})+2\Li_4(1-e^{-2z})}
\notag
\\&\quad+\frac{3}{8}\bra{
\Li_5(e^{-2z})+4 \Li_5(1-e^{2 z})+4\Li_5(1-e^{-2z})
}\notag
\end{align}
and analogous to \eqref{eq:ls-3-0}, \eqref{eq:ls-4-0}, we have
\begin{align}
\label{eq:lsh-3-0}
\lsh_3^{(0)}(z) &= \frac{z^3}{3}+\frac12 \bra{ \Li_3(1-e^{2 z})-\Li_3(1-e^{-2z})}
\displaybreak[1]
\\
\label{eq:lsh-4-0}
\lsh_4^{(0)}(z)&= \frac{3}{4}\bra{\Li_4(e^{-2z})-2 \Li_4(1-e^{2 z})+2\Li_4(1-e^{-2z})}+\frac{z^4}{4}-z^3 \log(2 \sinh z)
\\&\quad-\frac34 \z(4) +\frac{3z}{2}\z(3)-\frac{3 z^2}{2}\z(2)\notag
\displaybreak[1]
\end{align}
\end{corollary}
All of these are proved in exactly the same way as their trigonometric counterparts in the previous sections (via Theorem~\ref{thm:Li_m+1,k-hyperbolic}), where we note that log-sinh integrals of the form $\lsh_m^{(m-2)}(z)$ are evaluated using the series expansion
\begin{equation}\label{eq:log-sinh-fourier}
\log(2\sinh t)=t-\sum_{n=1}^\infty \frac{e^{-2 t n}}{n}.
\end{equation}
By the Taylor expansion of $\operatorname{arcsinh}^2$ we obtain the hyperbolic analogue of \eqref{eq:ls-m-1-series},
\begin{equation}\label{eq:lsh-m-1-series}
\lsh^{(1)}_m(z)= \frac{(-1)^{m-1}(m-2)!}{2^m}
\sum_{n=1}^\infty \frac{(-1)^{n-1}(2 \sinh z)^{2n}}{n^m \binom{2n}{n}}, \qquad (m\geq 2),
\end{equation}
(with an analogous formula for the $\lsh_m^{(0)}(z)$). Thus, setting
$z=\frac12\log(2)$ in equations \eqref{eq:lsh-3-1}--\eqref{eq:lsh-5-1}, we immediately recover the rapidly-converging series
\begin{align}
\sum_{n=1}^\infty \frac{(-1)^{n-1}}{n^3 2^n \binom{2n}{n}} &= \frac{\z(3)}{4}-\frac{\log(2)^3}{6}
\displaybreak[1]
\\
\sum_{n=1}^\infty \frac{(-1)^{n-1}}{n^4 2^n \binom{2n}{n}} &=
4 \Li_4\left(\tfrac{1}{2}\right)+\frac{13}{4} \zeta (3) \log (2)-\frac{7}{2}\zeta(4)+\frac{5 \log ^4(2)}{24}-\zeta(2) \log ^2(2)
\displaybreak[1]
\\
\sum_{n=1}^\infty \frac{(-1)^{n-1}}{n^5 2^n \binom{2n}{n}} &=
10 \Li_5\left(\tfrac{1}{2}\right)+6 \Li_4\left(\tfrac{1}{2}\right) \log (2)-\frac{19 \zeta (5)}{2}-\frac{2}{3} \zeta (2) \log ^3(2)+\zeta (3) \log^2(2)
\\&\quad+\frac{7}{2} \zeta (4) \log (2)+\frac{19 \log ^5(2)}{120}\notag
\end{align}
which were also given by Au in \cite{au20} using the theory of multiple-zeta values (MZV). Our method instantly produces many exotic rapidly-converging series. For example, setting $z=\frac12 \log(\frac{3}{2})$ in \eqref{eq:lsh-4-1}
leads to
\begin{equation}
\sum_{n=1}^\infty \frac{(-1)^{n-1}}{n^4 6^n \binom{2n}{n}} =
4 \Li_4\left(-\tfrac{1}{2}\right)+4\Li_4\left(\tfrac{1}{3}\right)+2\log \bra{\tfrac{3}{2}}\bra{ \Li_3\left(\tfrac{1}{3}\right)-\Li_3\left(-\tfrac{1}{2}\right)} -\frac{1}{8} \log^4\left(\tfrac{3}{2}\right)
\end{equation}
whereas setting $z=\frac12 \log(\phi)$, with $\phi\coloneqq \frac{1+\sqrt{5}}{2}$ the golden ratio, leads to
\begin{equation}
\sum_{n=1}^\infty \frac{(-1)^{n-1}}{n^4 (2+\sqrt{5})^n \binom{2n}{n}} =
4 \Li_4\left(\phi^{-2}\right)+4 \Li_4\left(-\phi^{-1}\right)-2\log (\phi ) \left(\Li_3\left(-\phi^{-1}\right)-\Li_3\left(\phi^{-2}\right)\right) -\frac{1}{8} \log ^4(\phi ).
\end{equation}
\section{Functional equations for Nielsen polylogarithms}\label{sec:nielsen-functional-equations}
We can also exploit Theorem~\ref{thm:Li_m+1,k-hyperbolic} to obtain functional equations for Nielsen polylogarithms. All of the following identities have already been proved (in an equivalent form) by Charlton et al. in \cite{cgr21} using Goncharov's theory of motivic iterated integrals. Here, we show how these functional equations follow almost immediately from Theorem~\ref{thm:Li_m+1,k-hyperbolic}.
\begin{lemma}\label{lem:s32}
The Nielsen polylogarithm $S_{3,2}(z)$ satisfies the following identity:
 \begin{align*}
 \begin{split}
S_{3,2}(1-e^{2z})+S_{3,2}(1-e^{-2z}) &=
 \bra{2\Li_5(1-e^{2z})+2\Li_5(1-e^{-2z})-\Li_5(e^{-2z})}
 \\&\quad-2z\bra{\Li_4(1-e^{2z})-\Li_4(1-e^{-2z})}
 +\frac{2z^5}{15}
 \\&\quad-\frac{2z^4}{3}\log(2\sinh z)+\z(5)-2z \z(4)+2z^2\z(3)-\frac{4z^3}{3}\z(2)
 \end{split}
 \end{align*}
\end{lemma}
\begin{proof} By Theorem~\ref{thm:Li_m+1,k-hyperbolic} we have
 \begin{align*}
 2 z \Li_4(1-e^{2z})+S_{3,2}(1-e^{2z}) &= \frac23 \int_0^z (z-t)\bra{\tlgsh{t}{z}-(z-t)}^3dt,
 \\
  -2 z \Li_4(1-e^{-2z})+S_{3,2}(1-e^{-2z}) &= \frac23 \int_0^z (z-t)\bra{\tlgsh{t}{z}+(z-t)}^3dt.
 \end{align*}
 Adding the above two equations gives
\begin{align}\label{eq:S23-idnetity-01}
\begin{split}
2z\bra{\Li_4(1-e^{2z})-\Li_4(1-e^{-2z})}+S_{3,2}(1-e^{2z})+S_{3,2}(1-e^{-2z})\\ =\frac{4}{3}\int_0^z \sqbra{ (z-t)\tlgsh{t}{z}^3+3(z-t)^3 \tlgsh{t}{z} }dt
\end{split}
\end{align}
Meanwhile, another application of Theorem~\ref{thm:Li_m+1,k-hyperbolic} yields
\begin{equation}\label{eq:S23-idnetity-02}
\Li_5(1-e^{2z})+\Li_5(1-e^{-2z}) =\frac{2}{3}\int_0^z\sqbra{ (z-t)\tlgsh{t}{z}^3+(z-t)^3 \tlgsh{t}{z} }dt
\end{equation}
The desired functional equation is obtained by eliminating the $\int_0^z (z-t)\tlgsh{t}{z}^3 dt$ term between equations
\eqref{eq:S23-idnetity-01}--\eqref{eq:S23-idnetity-02} and employing the series expansion \eqref{eq:log-sinh-fourier} for $\log(\sinh t)$. 
\end{proof}

By a similar argument we obtain a three-term functional equation for $S_{3,2}$:
\begin{lemma}\label{lem:s32-three-term}
	The Nielsen polylogarithm $S_{3,2}(z)$ satisfies the following identity:
	\begin{align*}
	\begin{split}
	S_{3,2}(1-e^{2z})&-S_{3,2}(1-e^{-2z})-2S_{3,2}(e^{-2z}) 
	\\&\hspace{-3em}= -3\Li_5(e^{-2z})-2z\bra{\Li_4(1-e^{2z})+\Li_4(1-e^{-2z})+\Li_4(e^{-2z})}+2\log(2\sinh z)\Li_4(e^{-2z})
	\\&\hspace{-2em}-
	\frac{2 z^5}{5}-\frac{4z^3}{3}\log(2\sinh z)^2+2\log(2\sinh z)\bra{
\frac{z^4}{3}-\zeta(4)+2z \z(3)-2z^2 \z(2)}
	\\&\hspace{-2em}-2z^2\z(3)-3z\z(4)+2\z(2)\z(3)-\z(5)
	\end{split}
	\end{align*}
\end{lemma}
\begin{proof}
The proof is similar to the proof of Lemma~\ref{lem:s32}, only slightly more involved. For ease of notation, let us denote 
$\Li_m^\pm\coloneqq \Li_m(1-e^{\pm 2 z})$ (and likewise for $S_{m,k}^\pm$), and denote $T\coloneqq (z-t)$, $L\coloneqq \tlgsh{t}{z}$. From Theorem~\ref{thm:Li_m+1,k-hyperbolic}, it  immediately follows that
\begin{align*}
-(S_{3,2}^{+}-S_{3,2}^{-})&-2z(\Li_4^{+}+\Li_4^{-})
= \frac{2}{3}\int_0^z\bra{T(L+T)^3-T(L-T)^3}dt
\\&=\frac{2}{3}\int_0^z \bra{6T^2L^2 + 2T^5}dt
=\frac{4 z^5}{15}+4\int_0^z (z-t)^2 \tlgsh{t}{z}^2dt.
\end{align*}
The last integral in the second line can be evaluated using the series expansion
\begin{equation}\label{eq:log-sinh^2}
\log(2\sinh t)^2=\bra{t+\log(1-e^{-2t})}^2= t^2-2t \sum_{n=1}^\infty\frac{e^{-2tn}}{n}+2\sum_{n=1}^\infty \frac{H_{n-1}}{n}e^{-2 tn},
\end{equation} 
which follows from \eqref{eq:Li_1..1-log}. Upon simplification we obtain the desired functional equation, where the only non-trivial contribution comes from the term
\[
\int_0^z (z-t)^2\sum_{n=1}^\infty\frac{H_{n-1}}{n}e^{-2 tn} dt
=\sum_{n=1}^\infty H_{n-1}\bra{ \frac{1- e^{-2 n z}}{2n^4}-\frac{3}{8 n^5}-\frac{ z}{n^3}+\frac{ z^2}{ n^2}}
\]
which introduces the $S_{3,2}(e^{-2z})$ term, and where the Euler sums  $S_{m-1,2}(1)=\sum_{n\geq 1}H_{n-1} n^{-m}$ are well-known\footnote{The investigation of such sums originated in a 1742 correspondence between Euler and Goldbach, cf. \cite[pp. 741--755]{euler_correspondence_2015}.} to be expressible in terms of integer values of the Riemann zeta function.
\end{proof}
To illustrate the effectiveness of these functional equations, notice that setting $z=\frac12 \log 2$ in Lemmas~\ref{lem:s32} \& \ref{lem:s32-three-term} yields closed-form expressions (in terms of classical polylogarithms) for 
$S_{3,2}(-1)+S_{3,2}(1/2)$ and $S_{3,2}(-1)-2S_{3,2}(1/2)$ respectively, thus establishing at once closed-forms for both $S_{3,2}(-1)$ and $S_{3,2}(1/2)$.

The following three-term functional equation for $S_{4,2}$ is equivalent to the one given in
\cite[Proposition 29]{cgr21}.
\begin{lemma}\label{lem:s42}
The Nielsen polylogarithm $S_{4,2}(z)$ satisfies the following identity:
\begin{align*}
\begin{split}
S_{4,2}(1-e^{2z})&+S_{4,2}(1-e^{-2z})+S_{4,2}(e^{-2z}) 
\\&\hspace{-3em}= 2\bra{\Li_6(1-e^{2z})+\Li_6(1-e^{-2z})+\Li_6(e^{-2z})}
\\&\hspace{-2em}-z \bra{2\Li_5(1-e^{2z})-2\Li_5(1-e^{-2z})-\Li_5(e^{-2z})}-\log(2\sinh z)\Li_5(e^{-2 z})
\\&\hspace{-2em}
+\log(2\sinh z)\bra{
\frac{2 z^5}{15}+\z(5)-2z \z(4)+2z^2 \z(3)-\frac{4 z^3}{3}\z(2)}-\frac{z^4}{3}\log(2 \sinh z)^2
\\&\hspace{-2em}-\frac{z^6}{15}-\frac{2 z^3}{3} \z(3)-\frac{3 z^2}{2}\z(4)+2 z \z(2) \z(3)-z \z(5)-\frac{5}{4}\z(6)-\frac{\z(3)^2}{2}
\end{split}
\end{align*}
\end{lemma}
\begin{proof}
Again, denote 
$\Li_m^\pm\coloneqq \Li_m(1-e^{\pm 2 z})$ (and likewise for $S_{m,k}^\pm$), and denote $T\coloneqq (z-t)$, $L\coloneqq \tlgsh{t}{z}$. By Theorem~\ref{thm:Li_m+1,k-hyperbolic},
 \begin{align*}
 -(S_{4,2}^{+}+S_{4,2}^{-})&-2z(\Li_5^{+}-\Li_5^{-})+2(\Li_6^{+}+\Li_6^{-} )
 \\&= \int_0^z\sqbra{\frac{1}{6} \left(T (L-T)^4+T (L+T)^4\right)-\frac{2}{60} \left((L+T)^5-(L-T)^5\right)}dt
 \\&=\frac{4}{3}\int_0^z \bra{L^2 T^3 + \frac{1}{5}T^5}dt
 =\frac{2 z^6}{45}+\frac{4}{3}\int_0^z (z-t)^3 \tlgsh{t}{z}^2dt.
 \end{align*}
Exactly as in the proof of Lemma~\ref{lem:s32-three-term}, 
the last integral can be evaluated using \eqref{eq:log-sinh^2}. This time the only non-trivial contribution comes from 
\[
\int_0^z (z-t)^3\sum_{n=1}^\infty\frac{H_{n-1}}{n}e^{-2 tn} dt
=\sum_{n=1}^\infty H_{n-1}\bra{ \frac{3 e^{-2 n z}}{8 n^5}-\frac{3}{8 n^5}+\frac{3 z}{4n^4}-\frac{3 z^2}{4 n^3}+\frac{z^3}{2n^2} }
\]
which introduces the $S_{4,2}(e^{-2z})$ term in addition to the known Euler sums.
\end{proof}
We note that the same method may be applied to prove the $S_{4,3}$ reduction formula given in \cite[Appendix D]{cgr21} and further reduction formulas for higher weights.

\section{Discussion and further remarks}\label{sec:further}
This paper centers on log-sine integrals and their connection to Apéry-like sums. Ultimately, our results follow from setting $c=-1/2$ in Theorem~\ref{thm:sum-t-1-t}. 
There is much room for the exploration of other implications of Theorem~\ref{thm:sum-t-1-t} for general $c\in(-1,0)$.  In fact, as mentioned, we expect that the hypotheses assumed in Theorem~\ref{thm:sum-t-1-t} are sufficient but not necessary -- numerical computations suggest that the claim remains true for  $c>-1$, $b\in\mathbb{C}$,  $a\in \mathbb{C}\setminus\{-1,-2,\ldots\}$, and $|t|<1$ which also satisfies $|t(1-t)^c|<|c|^c/(c+1)^{c+1}$.  

Several remarkable identities may be deduced from Theorem~\ref{thm:sum-t-1-t}. For instance, setting ${b\coloneqq ac}$ yields the following family of identities, where the left-hand side is independent of $c\in(-1,0)$:
\[
\sum_{n=1}^\infty \frac{t^n}{n+a} = \sum_{n=1}^\infty t^n (1-t)^{c(a+n)}
\frac{(1+c(a+n))_{n-1}}{(n-1)! (n+a)}.
\]
We remark that the above identity bears a striking resemblance to the identities given by Saha \& Sinha in \cite{ss24} (cf. also \cite{rosen25}).
As an amusing example, letting $t\to -1$ in the above equation, expanding around $a=-1/2$ and and comparing coefficients of $(a+1/2)^1$, we obtain a conjectural\footnote{Strictly speaking, we have not proved the validity of Theorem~\ref{thm:sum-t-1-t} for $t=-1$.} family of infinite series converging to Catalan's constant $K=\beta(2)$ identically for $c\in(-1,0)$:
\begin{align}
\begin{split}
K&\coloneqq\sum_{n=0}^\infty \frac{(-1)^n}{(2n+1)^2}
\\&=
\sum_{n=0}^\infty \frac{(-1)^n2^{c\bra{n+\frac12}}}{(2n+1)^2} \bra{
1-c\bra{n+\tfrac12}\log(2)-\sum_{m=1}^n \frac{c\bra{n+\tfrac12}}{m+c\bra{n+\tfrac12}}
}\prod_{k=1}^n \bra{ 1+\frac{c\bra{n+\tfrac12}}{k}}
\end{split}
\end{align}
(in which the series in the first line is recovered by letting $c\to0$).

More generally, setting $b=a(c+1)-d$ in Theorem~\ref{thm:sum-t-1-t} for arbitrary $d\in \R$, we obtain the Saha-Sinha-like identity 
\begin{equation}
\frac{t}{1+a}\, {}_2F_1\bra{\begin{array}c 1,1+d\\2+a \end{array} \middle| \; t}=(1-t)^{a(c+1)-d} \sum_{n=1}^\infty \bra{t(1-t)^c}^n \frac{(1-d+a(c+1)+n c)_{n-1}}{(n-1)!(n+a)}
\end{equation}
for $t,a$ satisfying the hypotheses in Theorem~\ref{thm:sum-t-1-t}, where here again the left-hand side is constant with respect to $c\in(-1,0)$.

{\footnotesize
\bibliographystyle{unsrt}
\bibliography{paper_refs}	
}
\Address
\end{document}